\documentclass{article}
\usepackage[margin=3cm]{geometry}
\usepackage[utf8]{inputenc}
\usepackage[utf8]{inputenc}
\usepackage{amsmath, amsthm, amssymb, amsfonts}
\usepackage{dsfont}
\usepackage[T1]{fontenc}
\usepackage{pst-node}
\usepackage{tikz-cd} 
\usepackage[english]{babel} 
\usepackage{mathrsfs}
\usepackage{url}
\usepackage{setspace}
\usepackage{graphicx}
\graphicspath{ {./images/} }
\usepackage{amsmath}
\usepackage{enumitem}
\usepackage{tcolorbox}
\usepackage{xcolor}
\usepackage{hyperref}
\usepackage{stmaryrd}
\usepackage{upgreek}
\usepackage{mathtools}
\usepackage{filecontents}
\usepackage{float}
\theoremstyle{plain}
\newtheorem{thm}{Theorem}

\newtheorem{lemma}[thm]{Lemma}
\newtheorem{prop}[thm]{Proposition}
\newtheorem{corollary}[thm]{Corollary}

\theoremstyle{definition}
\newtheorem{defi}{Definition}

\newtheorem{example}{Example}
\newtheorem{question}{Question}

\newcommand{\numberset}{\mathbb}

\newcommand{\F}{\numberset{F}}
\newcommand{\Proj}{\numberset{P}}

\definecolor{bigout}{rgb}{0.79, 0.08, 0.48}
\definecolor{dad}{rgb}{0.0, 0.0, 0.9}

\title{On the existence of {\itshape Spot It!} decks that are not projective planes}
\author{Bianca Gouthier\footnote{Institut de Math\'ematiques de Bordeaux, 351, cours de la Libération - F 33 405 Talence, France, bianca.gouthier@math.u-bordeaux.fr }, Daniele Gouthier\footnote{
Scienza Express publishing house; Master in Science Communication "Franco Prattico", International School for Advanced Studies, Trieste, Italy, daniele@scienzaexpress.it}}
\date{}

\begin{document}

\maketitle

\begin{abstract}
    The game {\itshape Spot It!} is played with a deck of cards in which every pair of cards has exactly one matching symbol and the aim is to be the fastest at finding the match. It is known that finite projective planes correspond to decks in which every card contains $n$ symbols and every symbol appears on $n$ cards. In this paper we relax the hypothesis on the number of cards on which a symbol appears: we study symmetric decks in which every symbol appears the same number of times and we introduce the concept of {\itshape maximal} deck, providing a sufficient condition for a deck to have this property. We also produce various examples of interesting decks which do not correspond to projective planes.
\end{abstract}


\section{Introduction}

We consider two systems of axioms: the first one defines a {\itshape projective plane} while the second one defines a {\itshape Spot It!} deck. The game {\itshape Spot It!} is played with a deck of cards, each containing the same number of symbols (all distinct), in which every pair of cards has exactly one matching symbol (the one you should spot!), \cite{CogMay}, \cite{Polster}.
The two systems of axioms, as given in \cite{sengupta}, are the following.

A {\itshape projective plane} is a collection of lines and points satisfying the axioms:
\begin{enumerate}
    \item[PP1.] any two points are on exactly one line;
    \item[PP2.] any two lines have exactly one point in common;
    \item[PP3.] there exist four distinct points such that no three of them are on the same line.
\end{enumerate}
The last axiom is given in order to exclude some degenerate cases (such as the empty set, a single point, a single line etc).
Under these axioms there exist {\itshape finite projective planes}, the most elementary of which is the {\itshape Fano plane}, given by $7$ points and $7$ lines satisfying the above axioms. Here are two representations of it: in the first one rows represent lines and points are named with numbers from $1$ to $7,$ while the second is classical.
\begin{figure}[H]
\centering
\begin{minipage}[c]{.40\textwidth}
 \centering
    $
   \begin{array}{ccc}
5&1&2\\
5&3&4\\
6&1&3\\
6&2&4\\
7&1&4\\
7&2&3\\
  5&6&7\\
\end{array}$
\end{minipage}
\begin{minipage}[c]{.40\textwidth}\centering
\includegraphics[width=.50\textwidth]{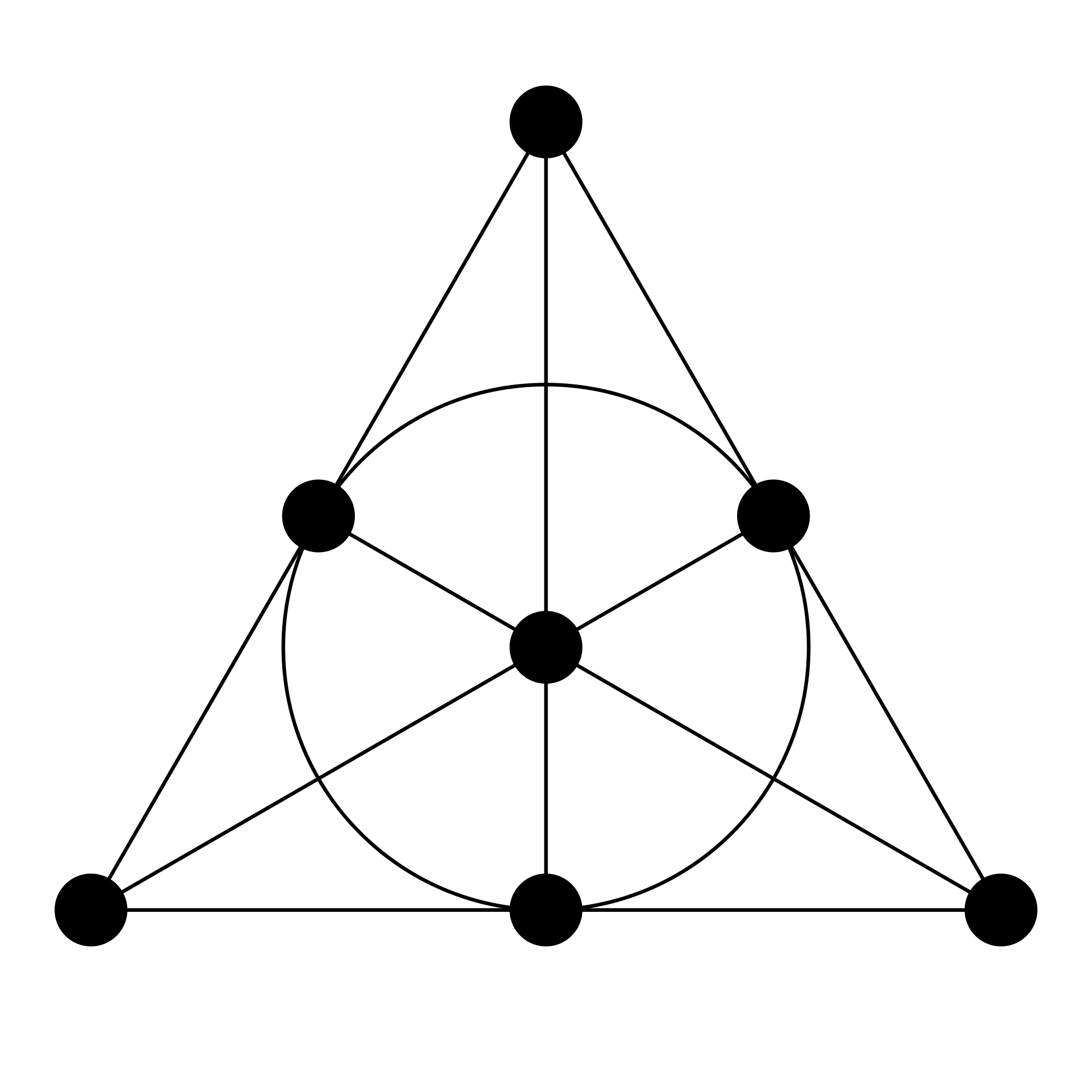}
\end{minipage}
\caption{Two depictions of the Fano plane}
\end{figure}

A {\itshape Spot It!} deck is a collection of cards and symbols satisfying the axioms:
\begin{enumerate}
    \item[D1.] any two cards have exactly one symbol in common;
    \item[D2.] every symbol is on at least two cards;
    \item[D3.] every card contains at least two symbols;
    \item[D4.] all cards contain the same number of symbols (all distinct);
    \item[D5.] at least one symbol exists.
\end{enumerate}

These are the axioms describing all the decks we will encounter in the following of this paper. Notice that axiom D2 tells us that there could not be a unique symbol common to all the cards of the deck: if this was the case, then all other symbols could appear only once and this is in contradiction with the quoted axiom. The non-existence of a symbol common to all the cards of the deck is equivalent to saying that for any symbol there exists a card not containing it.

Coggin and Mayer \cite{CogMay} showed that finite projective planes correspond to {\itshape Spot It!} decks. Thus, in the following of this paper, we will use as synonyms the words {\itshape point/symbol}, {\itshape line/card} and for instance we will say that two symbols are {\itshape aligned} if they belong to the same card. In addition to this, Sengupta \cite{sengupta} modifies the {\itshape Spot It!} axioms with the following (that, altogether, are equivalent to the projective plane axioms PP1, PP2 and PP3):

\begin{enumerate}
    \item[DP1.] any two cards have exactly one symbol in common, (D1);
    \item[DP2.] every symbol is on at least three cards;
    \item[DP3.] every card contains at least three symbols;
    \item[DP4.] all cards contain the same number of symbols, (D4);
    \item[DP5.] at least one symbol exists, (D5);
    \item[DP6.] any two symbols are on exactly one card.
\end{enumerate}
In this way {\itshape Spot It!} decks which are projective planes are described.
So, as already underlined, every finite projective plane is in particular a {\itshape Spot It!} deck, but there are decks which are not projective planes, since the new collection of axioms has stronger requests than the previous one. For example, axiom D2 asks that any symbol appears on at least two cards and not necessarily on $n$ as in a deck corresponding to a projective plane.

In this paper we will describe {\itshape symmetric} decks, in which every symbol is on the same number of cards; {\itshape paired} decks, in which every pair of symbols appears on a card (i.e. they satisfy DP6); and {\itshape maximal} decks, in which it is impossible to add a new card without adding new symbols. Paired decks correspond exactly to finite projective planes.

Via these three type of decks we will get some results on general decks and we will give examples of decks which are not projective planes.

\section{Multiplicities of symbols, order and length of a deck}

We begin by giving the following representation of a deck $D$: consider an integer $\ell,$ the set of symbols $S:=\{s\in\numberset{N}\mid s\leq \ell\}$ and a finite collection of subsets (the cards) of $S$ which realise the deck $D=\{C_1,\dots,C_c\}.$ Any two cards have the same cardinality $|C_i|=|C_j|=n$ (axiom D4) and exactly one symbol in common $|C_i\cap C_j|=1$ (axiom D1).

\begin{defi}
We define {\itshape order} of the deck the cardinality $n$ of each card and {\itshape length} of the deck the cardinality $\ell$ of the set of symbols. Moreover $c$ will denote the number of cards in the deck. 
\end{defi}

As in \cite{Jongsma}, we are interested in the relations between the order of the deck, its length and the number of cards.

\begin{lemma}[Uniqueness of a card]
For any two distinct symbols there is at most one card which contains them both. Thus the following inequality holds true: $$c\binom{n}{2}\leq\binom{\ell}{2}$$
and the equality holds if and only if any pair of distinct symbols belongs to a card, that is for any two distinct symbols there exists a card which contains them both.
\end{lemma}

\begin{proof}
The first statement is a straightforward consequence of axiom D1: indeed, otherwise we would find two cards with two symbols in common and this cannot happen. The inequality is just a translation of it: indeed, the left hand side counts the number of pairs of symbols (appearing together on some card) in a deck, while the right hand side counts the number of all possible pairs of symbols. Finally, we have the equality if all possible pairs of symbols appear in the deck.  
\end{proof}

It follows then a natural definition.

\begin{defi}
A deck is {\itshape paired} if any two symbols are on exactly one card, that is $c\binom{n}{2}=\binom{\ell}{2}.$
\end{defi} 

Section \ref{paireddecks} will be devoted to an overview on paired decks.

Following \cite{Heemstra}, given a symbol $s$ we are interested in the set of cards on which it appears and the cardinality of this set.

\begin{defi}
For any symbol $s$ we define the {\itshape star} of centre $s$ to be the set of cards on which $s$ appears: $E(s):=\{C\in D\mid C\ni s\}.$ The {\itshape multiplicity} $m(s)$ of the symbol $s$ is the number of cards on which it appears, that is the cardinality of the set of cards $E(s).$ 
\end{defi}

It is useful to keep in mind the following visualization of a deck. Pick any card $$C=\{s_1,\dots,s_n\}$$ and consider the symbols on it: you can partition the remaining $c-1$ cards in $n$ packs of cards, each of them collecting all the (remaining) cards which have in common one of the $n$ symbols on $C.$ Notice that each pack is exactly the star $E(s_i)$ with the card $C$ removed, $E(s_i)\backslash\{C\}$. Here is an example:
\begin{figure}[H]
    \centering
   \begin{align*}
\begin{array}{ccc}
  5&6&7\\
\end{array}&&
\begin{array}{ccc}
5&1&2\\
5&3&4\\
 \end{array}&&
 \begin{array}{ccc}
6&1&3\\
6&2&4\\
\end{array}&&
\begin{array}{ccc}
7&1&4\\
7&2&3
\end{array}
\end{align*} \caption{A partitioning of the deck corresponding to the Fano plane}
\end{figure}

Notice that by the axiom D2 we have that for any symbol $s$ the multiplicity is greater than or equal to $2,$ $m(s)\geq 2,$ since every symbol appears on at least two cards. Moreover we can count all the different symbols appearing in the star $E(s)$: indeed, each of the $m(s)$ cards in $E(s)$ contains $n$ symbols, of which exactly one is in common, that is $s.$ Therefore in the star $E(s)$ there appear exactly $m(s)(n-1)+1$ different symbols. Remark also that two stars $E(s_1)$ and $E(s_2)$ have at most one card in common, namely, if it exists, the card containing $s_1$ and $s_2.$

Let us consider the sum of the multiplicities of all the symbols in the deck. On one side, by definition, it is $\sum_{s\in S}m(s)$. On the other side, it is the product of the number of cards and the order of the deck, that is $$\sum_{s\in S}m(s)=cn.$$ Indeed both members count all the symbols which appear in the deck, counted, of course, with multiplicity.

\begin{lemma}[Sums of multiplicities]\label{sumsofmult}
For any fixed card $C$ in a deck, it holds $$\sum_{s\in C}m(s)=c+n-1.$$ Moreover, considering every symbol of the deck, it also holds $$\sum_{s\in S}m(s)^2=c(c+n-1).$$
\end{lemma}

\begin{proof}
Visualizing the deck partitioned in packs as we just described, one has that the number of cards is exactly $$c=\sum_{s\in C}(m(s)-1)+1,$$ where $C$ is any fixed card and thus the first equality follows. Now we sum it over all the cards of the deck. Since the second member does not depend on the choice of the card, we obtain $$\sum_{C\in D}\sum_{s\in C}m(s)=c(c+n-1).$$ In the sum of sums we are writing $m(s)$ times each multiplicity $m(s).$ Hence $$\sum_{C\in D}\sum_{s\in C}m(s)=\sum_{s\in S}m(s)^2$$ and also the second equality follows.
\end{proof}

The following result tells us that the algebraic mean of the multiplicities of all the symbols in a deck is less than or equal to the algebraic mean on a single card. This proof was pointed out to us by Francesco Viganò and simplified our reasoning in a previous version of this paper.

\begin{corollary}\label{inequality}
The following inequality holds true:
\begin{equation*}
   \frac{cn}{\ell}\leq\frac{c+n-1}{n}. 
\end{equation*}

\end{corollary}

\begin{proof}
If we view the multiplicity as a random variable, then the expected value is the given by the algebraic mean $$\frac{1}{\ell}\sum_{s\in S}m(s)=\frac{cn}{\ell}$$ while the variance is $$\frac{1}{\ell}\sum_{s\in S}\left(m(s)-\frac{cn}{\ell}\right)^2.$$ Using the well known fact that the variance is always non-negative, one finds the claimed inequality. Moreover, we see that the equality holds true if and only if $m(s)=\frac{cn}{\ell}$ for all $s\in S.$
\end{proof}

Looking at each single multiplicity we have two upper bounds.

\begin{thm}[Bounds for multiplicities]\label{boundsmult}
For any symbol $s,$ its multiplicity satisfies the following inequalities: $$m(s)\leq n$$ and $$m(s)\leq \frac{\ell-1}{n-1}.$$
\end{thm}

\begin{proof}
Assume that there exists a symbol $s$ appearing on $n+1$ cards and let $C$ be a card not containing $s$ (recall that such a card always exists thanks to axiom D2). Then $C$ has exactly one symbol in common with the $n+1$ cards considered and moreover these symbols must all be different since the cards already have $s$ in common. This leads to a contradiction since $C$ has exactly $n$ symbols and thus the first inequality holds.

We know that the star of centre $s$ contains $m(s)(n-1)+1$ symbols and this number cannot be greater than the length $\ell$ of the deck, hence the second inequality.
\end{proof}

Let us denote by $\mu$ and $M$ the minimum and the maximum of the multiplicities respectively: $$\mu:=\min_{s\in S}m(s),$$ $$M:=\max_{s\in S}m(s).$$

\begin{thm}[Bounds for the number of cards]
The following inequalities hold: $$c\leq\ell$$ and $$\mu\ell\leq cn\leq M\ell.$$
\end{thm}

\begin{proof}
We already know that $\sum_{s\in S}m(s)=cn$ and that $m(s)\leq n$ for every symbol $s.$ Thus $$\sum_{s\in S}m(s)\leq\ell n$$ and it follows that the number of cards is less than or equal to the number of symbols (the length) of the deck.

Starting from the same equality we may now use the relations $\mu\leq m(s)\leq M$ which hold true for every symbol $s,$ obtaining $\mu\ell\leq cn\leq M\ell$ as wished.

In some situations the following equivalent inequalities may be useful: $$\frac{cn}{M}\leq\ell\leq\frac{cn}{\mu}.$$
\end{proof}

Let us conclude this section with an example.

\begin{example}\label{2mult}
Consider a deck in which the symbols have just two different multiplicities $\mu$ and $M.$ Consider any card $C$ and let $n_\mu(C)$ and $n_M(C)$ be the number of symbols with multiplicity $\mu$ and $M$ on $C$ respectively. We thus have the equations $$\left\{\begin{array}{l}
     \mu n_\mu(C)+Mn_M(C)=c+n-1\\
     n_\mu(C)+n_M(C)=n
\end{array}\right.$$ whose solution is
$$\left\{\begin{array}{l}
     n_\mu(C)=\frac{Mn-c-n+1}{M-\mu}\\
     n_M(C)=\frac{c+n-1-\mu n}{M-\mu}
\end{array}\right.$$ which do not depend on the given card $C.$ So we may conclude that in a deck with just two multiplicities the number of symbols on a card with a fixed multiplicity is constant for all cards.
\end{example}

\section{The fundamental number of a deck}

In literature, it is well known that in a deck of order $n$ which is a finite projective plane there are exactly $n^2-n+1$ symbols (points) and $n^2-n+1$ cards (lines). Such a value plays a role even in the geometry of a generic deck (not necessarily a projective plane).

\begin{defi}
We will denote $\Delta_n$ the number $n^2-n+1$ and call it {\itshape fundamental number} of the deck.
\end{defi}

The fundamental number is related to the multiplicities of the symbols.

\begin{thm}\label{lenghteqfdnb}
If there exists a symbol $s$ with multiplicity equal to $n$, $m(s)=n,$ then the length of the deck is exactly its fundamental number $\ell=\Delta_n.$
\end{thm}

\begin{proof}
By the second inequality of Theorem \ref{boundsmult}, if $s$ is a symbol of multiplicity $n,$ then $$n=m(s)\leq\frac{\ell-1}{n-1}$$ and thus $$n(n-1)+1=\Delta_n\leq\ell.$$ Since $m(s)=n,$ the number of different symbols appearing in the star $E(s)$ is $\Delta_n.$
Now, take a card $C$ not containing $s$ (recall that such a card always exists by axiom D2 as explained in the introduction). This card intersects each of the $n$ cards of $E(s)$ in exactly one symbol and, since $C$ cannot intersect them in $s,$ it intersects each card of $E(s)$ in a different symbol. Therefore, the $n$ symbols appearing on $C$ were already listed in the star. Hence $\ell=\Delta_n,$ as wished.
\end{proof}

\begin{thm}
If $\ell>\Delta_n,$ then any symbol admits a not-aligned one. In particular it holds that $$c\binom{n}{2}<\binom{\ell}{2}.$$
\end{thm}

\begin{proof}
Suppose that $\ell>\Delta_n.$ For any symbol $s$ the star $E(s)$ contains $$m(s)(n-1)+1\leq\Delta_n<\ell$$ different symbols, that is there exists at least one symbol not aligned to $s.$ Therefore $$c\binom{n}{2}<\binom{\ell}{2}$$ since this inequality is equivalent to the existence of at least two non-aligned symbols.
\end{proof}

\begin{thm}
As above, let us denote $\mu$ and $M$ the minimum and the maximum multiplicities respectively. The following facts hold true:
\begin{enumerate}
    \item $n+1\leq n(\mu-1)+1\leq c\leq n(M-1)+1\leq\Delta_n;$
    \item $c\leq n(M-1)+1\leq M(n-1)+1\leq\ell;$
    \item there exists a multiplicity common to at least $n+1$ symbols.
\end{enumerate}
\end{thm}

\begin{proof}\leavevmode
\begin{enumerate}
    \item By the first equality of Lemma \ref{sumsofmult} we know that the number of cards is $$c=\sum_{s\in C}(m(s)-1)+1,$$ where $C$ is any fixed card. Then $1.$ follows by using the inequalities $2\leq\mu\leq m(s)\leq M\leq n$ for all $s.$
    \item Since $M\leq n,$ $$n(M-1)+1=nM-n+1\leq nM-M+1=M(n-1)+1$$ and the latter is the cardinality of a star whose centre has multiplicity $M$ and thus the claimed inequalities hold.
    \item Observe that $2\leq\mu\leq M,$ hence in a deck there exist at most $M-1$ different multiplicities. Since $n(M-1)+1\leq\ell$ then by the pigeonhole principle at least $n+1$ symbols must have the same multiplicity, henceforth $3.$
\end{enumerate}
\end{proof}

\begin{corollary}
One has the equality $M(n-1)+1=\ell$ if and only if $M=n$ and this is also equivalent to the existence of a symbol $\tilde{s}$ aligned to all the other symbols.
\end{corollary}

\begin{proof}
The if part is Theorem \ref{lenghteqfdnb}. Conversely, if $M<n,$ let $s$ be a symbol of multiplicity $M$ and $C$ be a card not belonging to $E(s),$ which has exactly $M(n-1)+1$ different symbols. The card $C$ intersects each of the $M$ cards of the star $E(s)$ in a different symbol and since on each card there are $n>M$ symbols, there must be a symbol on $C$ that did not already appear in $E(s).$
For the last statement, notice that a symbol $\tilde{s}$ is aligned to all the other symbols if and only if the star $E(\tilde{s})$ exhausts all the symbols, that is $m(\tilde{s})(n-1)+1=\ell.$
\end{proof}

We still don't know if one of the inequalities $m(s)\leq n$ and $m(s)\leq \frac{\ell-1}{n-1}$, for any symbol $s,$ is finer than the other. Equivalently, we are interested in knowing if there is a relation between the length of a deck and the fundamental number. We know that if there exists a symbol of multiplicity $n,$ then the length is equal to the fundamental number. Moreover we have the following:

\begin{lemma}
Take $kn+2$ cards in a deck of order $n$, with $k$ a positive integer. Then on each of these cards there exists a symbol belonging to $k+2$ of them.
\end{lemma}
\begin{proof}
Lets' pick a card $C$. Any of the others $kn+1$ intersects $C$ in exactly one symbol. So, by the pigeonhole principle there exist $k+1$ of these $kn+1$ cards intersecting $C$ in a same symbol, which belongs to $k+2$ cards in the picked ones.
\end{proof}

\begin{corollary}
If a deck of order $n$ has at least $n^2-2n+2$ cards, then on any card there is a symbol of multiplicity $n.$
\end{corollary}

Under this hypothesis the length is exactly equal to the fundamental number.
In this paper we provide some examples of decks: all of them have $\ell\leq\Delta_n$. So we propose the following:

\begin{question}\label{q1} We strongly suspect that the length should always be smaller than or equal to the fundamental number, but is it true? Can anybody exhibit a deck with length $\ell>\Delta_n$?
\end{question}

\section{Symmetric decks}

In literature, some authors (see for instance \cite{Gazagnes} or \cite{Heemstra}) focused on symmetric decks. In a symmetric deck all symbols have the same multiplicity. A finite projective plane is a deck in which all symbols have the same multiplicity $m(s)=n$, where $n$ is the order of the deck.

\begin{defi}
A deck is said to be $M$-symmetric if every symbol has multiplicity equal to $M.$
\end{defi}

\begin{example}[$2$-symmetric decks]\label{2sym}
Consider the matrix
$$
    \begin{array}{ccc}
    a&b&c\\
    a&d&e\\
    b&d&f\\
    c&e&f
    \end{array}
$$
in which the values $a,b,c,d,e,f$ are different. Every row is a card and the set of $4$ cards is a $2$-symmetric deck of order $3,$ length $6$.

In a similar fashion we may construct a $2$-symmetric deck of order $n$ and length $\binom{n+1}{2},$ formed by $n+1$ cards: start with an upper diagonal $n\times n$ matrix with all $\binom{n+1}{2}$ different entries, transpose it and glue it below to get the $n+1$ cards.
\begin{figure}[H]
    \centering
   
   $
   \begin{array}{ccc}
    a_{11}&\dots&a_{1n}\\
    a_{11}&\ddots&\vdots\\
    \vdots&\ddots&a_{nn}\\
    a_{1n}&\dots&a_{nn}
    \end{array}
$ \caption{$2$-symmetric deck of order $n$}
   
\end{figure}
\end{example}

\begin{thm}[Characterisation of symmetric decks]\label{symdecks}
Let $M$ be the highest multiplicity of the deck. Then the following facts are equivalent:
\begin{enumerate}
    \item the deck is $M$-symmetric;
    \item $cn=\ell M;$
    \item $c=n(M-1)+1;$
    \item there exists a card $C$ such that every symbol on it has multiplicity $M.$
\end{enumerate}
\end{thm}

\begin{proof}
If the deck is $M$-symmetric, then the sum of the multiplicities of all the symbols is both $cn$ and $\ell M,$ thus $cn=\ell M.$ Moreover, the sum of the multiplicities on a given card is both $c+n-1$ and $nM,$ henceforth $c=n(M-1)+1.$ Finally, on every card every symbol has multiplicity $M.$ So the first assertion implies all the others.

Now, if $cn=\ell M,$ then the sum of the multiplicities of all symbols is $$\sum_{s\in S}m(s)=\ell M.$$ But this sum has $\ell$ terms which are all less than or equal to $M.$ Therefore all of them must be equal to $M$ and thus the deck is $M$-symmetric (and $3.$ and $4.$ follow too).

If $c=n(M-1)+1,$ arguing as above (this time using that $\sum_{s\in C}m(s)=nM$ for any card $C$) we find that every symbol on every card has multiplicity $M,$ that is the deck is $M$-symmetric (and $2.$ and $4.$ follow too).

If there exists a card $C$ on which every symbol has multiplicity $M,$ then $c=n(M-1)+1$ (and $1.$ and $2.$ follow too).
\end{proof}

One of the proposed mini-games in Dobble\footnote{Dobble is the name used in Europe for the game {\itshape Spot It!}, which is instead originally known like this in the US. The distributed game is a paired deck of order $8$ (in fact two cards are missing, there are 55 cards instead of 57, but this is, apparently, just due to commercial reasons).}, asks to lay out $9$ cards and find $3$ cards with a common symbol. The first author discovered this mini-game together with Margherita Pagano and it felt natural to wonder why such three cards exist. Fruitful discussions with Margherita have been of inspiration for the following example.

\begin{example}\label{hp} Take $n+1$ cards in a deck $D$ of order $n\geq4$, which is not $2$-symmetric. Then there exists at least one symbol belonging to at least three of the $n+1$ cards. In addition, there exists a symbol belonging to exactly one of them.
\end{example}
\begin{proof}
First of all observe that the chosen $n+1$ cards cannot form themselves a deck, indeed otherwise by the above characterization of symmetric decks they would form a $2$-symmetric deck, and one may verify that there cannot be a $2$-symmetric deck included in a deck of order $n\geq4$. Now, if we pick some cards from a deck these will always have pairwise a symbol in common, therefore the only axiom that could be contradicted is D2, that is there must be an isolated symbol. 

Now, suppose by contradiction that on the $n+1$ cards there are only symbols appearing once or twice (notice that they cannot all appear only once, since two cards have always a symbol in common). Consider then a symbol appearing only once and the card $C$ on which it appears. Then, the other $n$ cards of the pack intersect $C$ in the remaining $n-1$ symbols and thus, by the pigeonhole principle, one of the symbols must appear on the remaining cards twice, that is in total it appears three times.
\end{proof}

Notice that for $n=3$ the above argument fails: indeed the $2$-symmetric deck $$
    \begin{array}{ccc}
    a&b&c\\
    a&d&e\\
    b&d&f\\
    c&e&f
    \end{array}
$$ is included in the $3$-symmetric deck $$
    \begin{array}{ccc}
    a&b&c\\
    a&d&e\\
    b&d&f\\
    c&e&f\\
    g&a&f\\
    g&b&e\\
    g&c&d
    \end{array}
$$

\begin{corollary}
In a symmetric deck, the length is less than or equal to the fundamental number.
\end{corollary}

\begin{proof}
As already observed in Corollary \ref{inequality}, we have that a deck is symmetric if and only if $$\ell=\frac{cn^2}{c+n-1}.$$
With a straightforward computation one finds that $$\ell-\Delta_n\leq\frac{n-1}{c+n-1}(c-\Delta_n)$$ and using the already proven fact that $c\leq\Delta_n$ one obtains that also $\ell\leq\Delta_n.$ 
\end{proof}

\begin{corollary}\label{msym}
In an $M$-symmetric deck the following relation holds:
$$\binom{\ell}{2}-c\binom{n}{2}=\frac{(\ell-c)-(n-M)}{2}\ell.$$
\end{corollary}

\begin{proof}
Characterizations $2.$ and $3.$ of $M$-symmetric decks provide the wished result: $$\binom{\ell}{2}-c\binom{n}{2}=\frac{\ell(\ell-1)-cn(n-1)}{2}\stackrel{2.}{=}\frac{\ell(\ell-1)-\ell M(n-1)}{2}=$$$$\frac{\ell-1-M(n-1)}{2}\ell\stackrel{3.}{=}\frac{\ell-1-(c+n-1)+M}{2}\ell=\frac{(\ell-c)-(n-M)}{2}\ell.$$
\end{proof} 

In conclusion to this section, we ask in which orders $n$ we may have an $M$-symmetric deck. An answer is provided by putting together the symmetry characterizations $2.$ and $3.$ above, from which we obtain that $\ell M=n^2(M-1)+n,$ that is $$n^2=n\mod{M}.$$
Therefore, the orders $n$ in which an $M$-symmetric deck can exist are the idempotent elements in $\numberset{Z}/M\numberset{Z}.$
Equivalently, we can state that if $n$ is not idempotent modulo $M,$ then there cannot exist an $M$-symmetric deck of order $n.$ For instance, if $M$ is a prime number, then $\numberset{Z}/M\numberset{Z}$ is the field with $M$ elements and thus the only idempotent elements are $0$ and $1$: in this case, there cannot exist an $M$-symmetric deck of order $n$ with $n\neq 0,1\mod{M}.$

\section{Paired decks and finite projective planes}\label{paireddecks}

Recall that a deck is called {\itshape paired} if any two symbols are on exactly one card and that this is equivalent to the relation $c\binom{n}{2}=\binom{\ell}{2}.$

With the following two lemmas we prove that a paired deck is always symmetric.

\begin{lemma}
A deck of order $n$ in which the number of cards is exactly the fundamental number is $n$-symmetric.
\end{lemma}

\begin{proof}
For any card $C$ we have that $\sum_{s\in C}(m(s)-1)=c-1.$ Moreover, by hypothesis $$c-1=\Delta_n-1=n(n-1).$$ Since in the sum there are $n$ terms which are all less than or equal to $n-1,$ the equality holds if and only if $m(s)=n$ for all symbols, thus the deck is $n$-symmetric. 
\end{proof}

\begin{lemma}
A paired deck is $n$-symmetric.
\end{lemma}

\begin{proof}
Suppose that the deck is not symmetric, then there exists a symbol $s$ such that $m(s)<n.$ Let $C$ be a card not containing $s.$ If all the symbols on $C$ appeared together with $s,$ this would occur on $n$ different cards, that is $m(s)\geq n,$ which contradicts the assumption $m(s)<n.$ Therefore there is a symbol on $C$ that does never appear together with $s,$ that is the deck is not paired.
So we proved that a paired deck is symmetric. In particular it holds the equivalent condition $cn=\ell M.$ Moreover, $c+n=\ell+M$ by Corollary \ref{msym}. The system given by the two equations is uniquely solved by $\ell=c$ and $M=n,$ henceforth a paired deck is $n$-symmetric, as claimed.
\end{proof}

\begin{thm}[Characterization of paired decks]
The following facts for a deck of order $n$ are equivalent:
\begin{enumerate}
    \item the number of cards is the fundamental number;
    \item the deck is $n$-symmetric;
    \item the deck is paired.
\end{enumerate}
\end{thm}

\begin{proof}
The above lemmas assure that both $1.$ and $3.$ imply $2.$ Now, let us suppose that the deck is $n$-symmetric, then by the third characterization of symmetric decks we obtain that $c=\Delta_n,$ so $1.$ and $2.$ are equivalent and we are left to prove that they imply $3.$ Now, the first two together tell us that $c=\ell=\Delta_n$ and since $\Delta_n\binom{n}{2}=\binom{\Delta_n}{2},$ the deck is paired, as wished.
\end{proof}

A paired deck of order $n\geq 3$ satisfies the six axioms DP1,..., DP6 and thus is a finite projective plane.
Viceversa, every projective plane $\Proj^2_{\F_{p^k}}$ over any finite field $\F_{p^k}$ corresponds to a paired deck of order $n=p^k+1,$
see for example \cite{CogMay}.

\begin{example}\label{construction}
Let us give a construction for paired decks in which $n-1$ is prime,  construction which will be useful in order to obtain many more peculiar decks. Let us first do it concretely for $n=4$ and then explain how it works in general. What we will see in the end is in the spirit of the visualisation of a deck given in the first section, that is we will have $n$ blocks corresponding each to one symbol and a pivot card. Start with a $3\times 3$ matrix with all different symbols (beware that the rows are not yet the cards, indeed we are constructing a deck of order $4$):
$$
    \begin{array}{ccc}
       1  & 2&3 \\
        4 &5 &6\\
       7&8&9 
    \end{array}
$$
Transpose it:
$$
    \begin{array}{ccc}
       1  &4&7 \\
        2&5 &8\\
       3&6&9 
    \end{array}
$$
Now, starting again from the first matrix, take an element in the first row, descend of one line and advance of one column (we say that we move $1$-diagonally):
$$
    \begin{array}{ccc}
       1  & 5&9 \\
        2 &6 &7\\
       3&4&8 
    \end{array}
$$
We can do one more step: take an element in the first row, descend of one line and advance of two columns ($2$-diagonally):
$$
    \begin{array}{ccc}
       1  & 6&8\\
        2 &4 &9\\
       3&5&7
    \end{array}
$$
Now, we see that if we consider lines in different matrices, they have pairwise exactly one symbol in common (by construction!) but we notice that lines in the same matrix miss a common symbol, we should then add it in order to make them into cards:
$$
    \begin{array}{cccc}
      10& 1  & 2&3 \\
      10&  4 &5 &6\\
      10& 7&8&9 
    \end{array}
$$
$$
    \begin{array}{cccc}
     11&  1  &4&7 \\
        11&2&5 &8\\
       11&3&6&9 
    \end{array}
$$
$$
    \begin{array}{cccc}
      12& 1  & 5&9 \\
      12&  2 &6 &7\\
       12&3&4&8 
    \end{array}
$$
$$
    \begin{array}{cccc}
     13&  1  & 6&8\\
     13&   2 &4 &9\\
       13&3&5&7
    \end{array}
$$
We finally notice that we can still add our pivot card, namely the one given by the symbols which are the "centers" of each block we constructed:
$$\begin{array}{cccc}
     10&11&12&13
\end{array}$$
We thus obtained a paired deck of order $4.$
We can do the same construction for any $n$ such that $n-1$ is prime: we start with an $(n-1)\times(n-1)$ matrix with all different symbols, we transpose it and, to obtain the following matrices, each time we move one step further in choosing the next column, that is we move $k$-diagonally for every $k=1,\dots,n-2.$ We then add the common symbol for each of the $n$ obtained matrices and lastly we construct the card given by these $n$ symbols. This procedure provides us with a paired deck of order $n.$ Notice that this construction works for $n-1=p$ prime because all the numbers between $1$ and $p-1$ are coprime with $p$ and thus we always end up in different columns (otherwise we would find two symbols that were already together in the transposed the matrix). Saying it more mathematically, the method works because all the elements $1,2,\dots,p-1$ are generators of the group $\numberset{Z}/p\numberset{Z}.$

Observe also that the one described is the construction of the projective plane on the field $\F_p$ with $p=n-1$ elements. Each block constructed represents the parallel affine lines of the possible slopes (which are $0,1,\dots,p-1,\infty$), we then projectivise them adding the point at infinity in which each block of parallel lines meets and we finally add the line at infinity passing through all the points at infinity.
\end{example}

In general, what we can do to mimic this construction after the passage in which we transposed the matrix, is doing steps only of a certain pace: we will give an example and make this more precise after the following result, which tells us that there are certain values of $n$ for which there do not exist paired decks of order $n.$

\begin{thm}[Bruck-Ryser]
If $N$ is congruent to $1$ or $2$ $\mod 4$ and $N$ is not the sum of two squares, then there does not exist a projective plane of order $N.$
\end{thm}

This was proven for the first time in \cite{BrRy}. The smallest number for which the previous results don't give an answer about the existence of a paired deck of that order, is $n=11,$ but in \cite{Lam} it was proven that there does not exist a finite projective plane of order $10,$ and thus there does not exist a paired deck of order $11$ as well. There are conjectures saying that also projective planes of order $12$ do not exist, see for example \cite{Abdolhadi}.

So there are some orders in which no paired decks exist. For instance, by the above theorem, there does not exist a paired deck of order $7.$ We are nevertheless able to construct different decks of order $7$: by Example \ref{2sym} we know that there is a $2$-symmetric deck of order $7.$ Let us instead show that with the construction we described at the end of the previous example we can obtain a non-symmetric deck of order $7$:
\begin{figure}[H]
\centering
\begin{minipage}[c]{.40\textwidth}\centering
    $
\begin{array}{ccccccc}
     37&1&2&3&4&5&6  \\
     37&7&8&9&10&11&12\\
     37&13&14&15&16&17&18\\
     37&19&20&21&22&23&24\\
     37&25&26&27&28&29&30\\
     37&31&32&33&34&35&36
\end{array}$
\caption{Starting block}
\end{minipage}
\begin{minipage}[c]{.40\textwidth}\centering
$
\begin{array}{ccccccc}
     38&1&7&13&19&25&31  \\
     38&2&8&14&20&26&32\\
     38&3&9&15&21&27&33\\
     38&4&10&16&22&28&34\\
     38&5&11&17&23&29&35\\
     38&6&12&18&24&30&36
\end{array}$
\caption{Block obtained transposing}
\end{minipage}
\end{figure}
\begin{figure}[H]
\centering
$
\begin{array}{ccccccc}
     39&1&8&15&22&29&36  \\
     39&2&9&16&23&30&31\\
     39&3&10&17&24&25&32\\
     39&4&11&18&19&26&33\\
     39&5&12&13&20&27&34\\
     39&6&7&14&21&28&35
\end{array}$
\caption{Block obtained $1$-diagonally}
\end{figure}

As you see, after transposing the matrix, we did only a step of pace $1$. We obtained a non-symmetric deck of order $7,$ length $$\ell=(n-1)^2+3=39,$$ $$c=3(n-1)=18$$ cards and in which there are $(n-1)^2=36$ symbols of multiplicity $3$ and $3$ symbols of multiplicity $6$. In general, for $n-1$ not a prime number, we can construct with this method $p+1$ blocks, where $p$ is the smaller prime dividing $n-1$. In this way we obtain a non-symmetric deck of order $n,$ length $\ell=(n-1)^2+p+1,$ $c=(p+1)(n-1)$ cards and in which there are $(n-1)^2$ symbols of multiplicity $p+1$ and $p+1$ symbols of multiplicity $n-1$.

\section{Maximal decks}

In conclusion, we introduce maximal decks: these are decks to which it is not possible to add a new card without introducing new symbols.

\begin{defi}
A deck $D$ is said to be {\itshape maximal} if it is not possible to add a new card $\Gamma$ such that $D\cup\{\Gamma\}$ is still a deck. 
\end{defi}

Notice that all paired decks are maximal, indeed for a paired deck $c=\Delta_n$ and we cannot add any more cards since it always holds $c\leq\Delta_n.$ So we are interested in studying maximal decks which are not paired. 

\begin{prop}\label{maximal}
A deck $D$ is maximal if there do not exist $n$ symbols $\{s_1,\dots,s_n\}\subset S$ such that $\sum_{i=1}^nm(s_i)=c.$
\end{prop}

\begin{proof}
Suppose that $D$ is not maximal. Then, by definition, we are able to add a card to obtain a well defined deck, that is there exist $n$ symbols $S\supset\{s_1,\dots,s_n\}=:\Gamma$ such that $|\Gamma\cap C|=1$ for every $C\in D.$ Set $D_\Gamma:=D\cup\{\Gamma\}.$ Then, we have that the multiplicities in the new deck are
$$m_\Gamma(s)=\left\{\begin{array}{lc}
    m(s) & \mbox{if }s\not\in\Gamma \\
    m(s)+1 & \mbox{if }s\in\Gamma
\end{array}\right.$$ and the number of cards is $c_\Gamma=c+1.$ Summing the squares of the multiplicities on the new deck we have that $$\sum_{s\in S}m_\Gamma(s)^2=c_\Gamma(c_\Gamma+n-1)=(c+1)(c+n)$$ but we also have that $$\sum_{s\in S}m_\Gamma(s)^2=\sum_{s\not\in\Gamma}m(s)^2+\sum_{s\in\Gamma}(m(s)+1)^2=\sum_{s\in S}m(s)^2+2\sum_{s\in\Gamma}m(s)+n=c(c+n-1)+2\sum_{s\in\Gamma}m(s)+n$$ and equalizing the two quantities we obtain $2\sum_{s\in\Gamma}m(s)=2c,$ from which the result.
\end{proof}

The above proof may be seen more visually in the following way: if the deck $D$ is not maximal, then we are able to find $n$ symbols $s_1,\dots,s_n\in S$ that pairwise do not appear on any card.
Then, the $n$ stars $E(s_i)$ give a partitioning of the deck: indeed they do not intersect by assumption (since two of the symbols never appear on the same card) and they cover all the deck since $D\cup\Gamma,$ with $\Gamma=\{s_1,\dots,s_n\},$ is still a deck, that is $\Gamma$ must intersect every other card, or in other words, one (and only one) of the $s_i$'s appears on each card of $D$. Hence, $\sum_{i=1}^nm(s_i)=c.$

We have the following immediate corollary.
\begin{corollary}
A deck $D$ is maximal if for every $n$-tuple of symbols $\{s_1,\dots,s_n\}\subset S,$ $\sum_{i=1}^nm(s_i)>c.$
\end{corollary}

\begin{prop}
Symmetric decks are maximal.
\end{prop}

\begin{proof}
For every $n$-tuple of symbols in an $M$-symmetric deck $$\sum_{i=1}^nm(s_i)=Mn>c=Mn-n+1.$$
\end{proof}

\begin{example}
In any order $n,$ there exist non-symmetric maximal decks.
\end{example}

\begin{proof}
We will deal separately the case in which $n-1$ is prime from the general one.
Let us start supposing that $n-1$ is prime. Recall the construction of Example \ref{construction}. If we take any $k$ blocks, with $2\leq k\leq n-1,$ we obtain a maximal deck. The fact that it is a deck holds true since every time you remove cards from a deck (without leaving some isolated symbols) you still obtain a deck. In this deck there are $\ell=\Delta_n-(n-k)$ symbols, precisely $k$ symbols of multiplicity $n-1$ and $(n-1)^2$ symbols of multiplicity $k,$ and $c=(n-1)k$ cards. Therefore, for any $n$ symbols picked we have that $$\sum_{i=1}^nm(s_i)\geq kn>k(n-1)=c.$$ For $k=n-1,$ what we obtain is an $(n-1)$-symmetric deck of order $n,$ while for all the other choices of $k$ the deck constructed is maximal but not symmetric.
Finally, notice that in all the blocks that we did not consider there are new symbols, and thus we cannot re-add the same cards of the construction of Example \ref{construction} without introducing new symbols.

Suppose now that $n-1$ is not prime. Then, we already saw that with our usual construction we obtain a deck that we partition in $p+1$ blocks where $p$ is the smaller prime dividing $n-1$, that is we have $p+1$ symbols of multiplicity $n-1$ and $(n-1)^2$ symbols of multiplicity $p+1$ and moreover the number of cards is $c=(n-1)(p+1).$ This is a maximal deck: indeed for any $n$ symbols chosen $$\sum_{i=1}^nm(s_i)\geq n(p+1)>c.$$ Notice that the same holds true if we choose $k$ of the $p+1$ block with $2\leq k\leq p+1.$

Observe moreover that these are examples of decks with only two multiplicities and with the formulas of Example \ref{2mult} we recover that on each card there is exactly one symbol of maximal multiplicity and $n-1$ symbols of minimal multiplicity, as expected by the previous computation.
\end{proof}

Let us conclude with a remark on the above example. The maximal decks constructed for $n-1$ a prime power $p^k$ are included, up to adding the necessary symbols, in $n$-paired decks, which are projective planes on $\F_{p^k}$ (the construction of which is well explained for example in \cite{CogMay}). On the other hand, the ones obtained in the non-prime-power case are totally new!

\section{Conclusions and questions}

In conclusion, we may say that there is a hierarchy in the set of {\itshape Spot It!} decks we analyzed: finite projective planes coincide with paired, or equivalently $n$-symmetric, decks and their set is strictly included in the one of symmetric decks, which is in turn strictly included in the set of maximal decks, included in the set of generic ones.

In literature it was known that finite projective planes correspond to some decks and that these cannot exist for certain orders. In this paper we pointed out that there is a bijection between finite projective planes and paired decks. We also characterized symmetric decks and introduced maximal decks, providing a sufficient condition of maximality. We moreover showed examples of non-paired symmetric decks and examples of non-symmetric maximal decks.
In addition to Question \ref{q1}, here are some other queries which remain open for us. 

\begin{question}
In this paper we focused on the existence of certain decks, while we did not touch the question of uniqueness: is it true that order, number of cards and length $(n,c,\ell)$ determine uniquely a deck? Or do there exist substantially different decks (meaning for example that the multiplicities of the symbols are differently distributed) with the same triple?
\end{question}

Fun fact: it is very common for games to allow expansion packs. As Matteo Silimbani pointed out to us, we proved that there cannot exist an expansion of the game Dobble (up to adding the two missing cards, of course), indeed, the deck complete of its $57$ cards is a maximal deck and thus cannot be expanded to a bigger deck.

\renewcommand{\abstractname}{Acknowledgements} \begin{abstract}The first author would like to thank Fabio Buccoliero, Margherita Pagano and Francesco Viganò for the fruitful discussions on this fun-math topic. We are also grateful to them for their useful remarks on a first draft of this paper. The second author would like to thank Alberto Saracco and Matteo Silimbani, because in a sunny spring day in April 2021 they started an engaging discussion on Facebook about {\itshape Spot It!}; we also thank them for their useful remarks on a first draft of this paper. We both thank our beloved mother and wife, Monica Conte, for the idea of the construction of some paired decks (see Example \ref{construction}). \end{abstract}

\bibliographystyle{alpha}
\bibliography{bib}

\end{document}